\documentclass{amsart}
\usepackage{cancel}
\usepackage{amssymb, amsmath, color}
\usepackage{verbatim}
\usepackage{mathrsfs}
\usepackage{amsmath}
\usepackage{amsfonts}
\usepackage{amssymb}
\usepackage{graphicx}
\usepackage{color}
\usepackage{epstopdf} 
\usepackage{epsfig}
\usepackage{tikz}
\usepackage[dvipsnames]{xcolor}
\usepackage{enumitem}
\usepackage{lineno}
\usepackage[hidelinks]{hyperref}

\DeclareMathOperator{\sdi}{{\rm sd}_{\iota}}

\newtheorem{thm}{Theorem}[section]

\newtheorem{lem}[thm]{Lemma}
\newtheorem{conj}[thm]{Conjecture}
\newtheorem{prop}[thm]{Proposition}
\newtheorem{wn}[thm]{Corollary}
\newtheorem{obs}[thm]{Observation}

\newtheorem{df}[thm]{Definition}
\newtheorem{cons}[thm]{Construction}

\newtheorem{claim}{Claim}[section]

\newcommand{\sta}{{\rm sta}}
\newcommand{\sd}{{\rm sd}}
\newcommand{\Fi}{\mathcal{F}_\iota}

\usepackage{color}

\newcommand{\diam}{\mathrm{diam}}

\usetikzlibrary{backgrounds}

\tikzset{
    v/.style={draw,fill,circle,inner sep=1.5}, 
    e/.style={draw, thick}
}
\tikzset{
    highlight/.style={line width=8pt, gray!50,shorten >=-2pt, shorten <=-2pt}
}

\title{Isolation critical graphs \\ under multiple edge subdivision}

\author{Karl Bartolo} 
\address{Department of Mathematics, University of Malta, Malta}
\email{karl.bartolo.16@um.edu.mt}

\author{Peter Borg}
\address{Department of Mathematics, University of Malta, Malta}
\email{peter.borg@um.edu.mt}

\author{Magda Dettlaff}
\address{Institute of Informatics, University of Gdańsk, Poland}
\email{magda.dettlaff@ug.edu.pl}

\author{Magdalena Lema\'nska}
\address{Institute of Applied Mathematics, Gdańsk University of Technology, Poland}
\email{magleman@pg.edu.pl}

\author{Pawe{\l} \.Zyli\'nski}
\address{Institute of Informatics, University of Gdańsk, Poland}
\email{pawel.zylinski@ug.edu.pl}
\date{}

\renewcommand{\shorttitle}{Isolation critical graphs under multiple edge subdivision}

\begin{document}
\begin{abstract}
This paper introduces the notion of an $(\iota,q)$-critical graph. 
The isolation number of a graph $G$, denoted by $\iota(G)$ and also known as the vertex-edge domination number of $G$, is the size of a smallest subset $D$ of the vertex set of $G$ such that the subgraph induced by the set of vertices that are not in the closed neighbourhood of $D$ has no edges. A graph $G$ is $(\iota,q)$-critical 
if 
every subdivision of $q$ edges of $G$ gives a graph whose isolation number is greater than $\iota(G)$, and $G$ has $q-1$ edges such that subdividing them gives a graph whose isolation number is $\iota(G)$. 
We show that an $(\iota,q)$-critical graph exists for every integer $q \ge 1$. We prove that if $G$ is a connected $m$-edge non-star graph, then $G$ is $(\iota,q)$-critical for some $q \le m - 1$. We show that this bound is best possible. 
We provide a general characterization of $(\iota,1)$-critical graphs as well as a constructive characterization of $(\iota,1)$-critical trees, demonstrating that $(\iota,1)$-criticality can be checked in linear time for trees.
\\[5mm]
{\AmS \; Subject Classification:} 05C05, 05C35, 05C69.
\\[2mm]
\textbf{Keywords}: isolation number, edge subdivision, critical graphs.
\end{abstract}

\maketitle

\section{Introduction}

Let $G =(V(G),E(G))$ be a connected graph of order $n = |V(G)|\geq 1$ and size $m = |E(G)|\geq 0$. For $S\subseteq V(G)$, let $N[S]$ denote the \emph{closed neighbourhood of $S$} (the set of vertices in $S$ and their neighbours). A subset $D$ of $V(G)$ is a \emph{dominating set of $G$} if $V(G)=N[D]$ \cite{HHH23,fund}, and $\gamma(G)$ denotes the size of a smallest dominating set of $G$. The concept of isolation arises by relaxing the condition $V(G)=N[D]$. It was introduced by Caro and Hansberg in~\cite{CaHa17} and then extensively studied in~\cite{Borg1, Borg2, Borg3, Borgrsc, BFK, BFK2, BK23, BG24, BGH24, ChCZ25, FK21, HAV25, LMSS24}, from different perspectives, to mention just a few.
Concretely, let $\mathcal{F}$ be a set of graphs. We say that $D$ is an \emph{$\mathcal{F}$-isolating of $G$} if no subgraph of $G-N[D]$ is a copy of a graph in $\mathcal{F}$. The \emph{$\mathcal{F}$-isolation number of $G$}, denoted by $\iota(G, \mathcal{F})$, is the size of a smallest $\mathcal{F}$-isolating set of $G$. If $\mathcal{F} = \{F\}$, then we may replace $\mathcal{F}$ by $F$ in the defined terms and notation. If $\mathcal{F} = \{K_2\}$, then the terms $\mathcal{F}$-isolating set and $\mathcal{F}$-isolation number, and the notation $\iota(G,\mathcal{F})$, are abbreviated to \emph{isolating set}, \emph{isolation number}, and $\iota(G)$, respectively. An isolating set is also called a \emph{vertex-edge dominating set}~\cite{BChHH15, B25, KMY21, KVK14,  L07, LHHF10, P86,sentUVED, ZZ20, Z19}.  An isolating set of $G$ of size $\iota(G)$ will be called a \emph{minimum isolating set} of $G$ or an \emph{$\iota(G)$-set}. Note that $D$ is a dominating set of $G$ if and only if it is a $K_1$-isolating set of $G$ (that is, $G-N[D]$ has no vertices). Thus, $\gamma(G) = \iota(G,K_1)$. 

For a graph $G$, the \emph{subdivision} of an edge $e = uv$ with a new vertex $w$, called the \emph{subdivision vertex}, is the operation that yields the graph $G_e$ with $V(G_e)=V(G)\cup \{w\}$ and $E(G_e)=(E(G)\setminus\{uv\})\cup \{uw,wv\}$. For $A \subseteq E(G)$, the (subdivided) graph obtained from $G$ by subdividing the edges in $A$ is denoted by $G_A$. The \emph{isolation subdivision number of $G$}, denoted by $\sdi(G)$, is the size of a smallest subset $A$ of $E(G)$ such that $\iota(G_A) > \iota(G)$. If $k \geq 0$ and $G$ is a copy of $K_{1,k}$, then $G$ is called a \emph{$k$-star} or simply a \emph{star}. A graph that is not a star will be called a \emph{non-star graph}. Since the isolation number of a star does not increase when all of its edges are subdivided, we therefore consider only connected non-star graphs. To the best of our knowledge, the only paper concerning the isolation subdivision number is~\cite{DLMZZ26}, where the authors establish that $1 \le \sd_{\iota}(T) \le 4$ for any non-star tree $T$, and provide a complete characterization of the class of trees whose isolation subdivision number is~$1$. Herein, we continue their study and focus on $(\iota,q)$-critical graphs. 

For $q \geq 1$, a graph $G$ is \emph{$(\iota,q)$-critical} if $\iota(G_A) > \iota(G)$ for each $A \subseteq E(G)$ with $|A|=q$, and $\iota(G_{A'}) = \iota(G)$ for some $A' \subseteq E(G)$ with $|A'|=q-1$. Thus, if $G$ is \emph{$(\iota,q)$-critical}, then $q$ is the smallest positive integer such that every subdivision of $q$ edges of $G$ gives a graph whose isolation number is greater than $\iota(G)$. 
To show that $q \leq m-1$ for a connected $m$-edge non-star graph, we prove a stronger statement and establish that this bound is best possible. For this, 
we construct a graph $I_k$ for $k \geq 1$ in Construction~\ref{cons-I}, and a graph $M_k$ for $k \geq 3$ in Construction~\ref{cons-M}. The following is our first main result, proved in Section~\ref{sec:general}.

\begin{thm}\label{thm-subdiv}
    If $G$ is a connected $m$-edge non-star graph, $A \subseteq E(G)$ and $|A| \geq m - 1$, then \[\iota(G_A) \geq \iota(G) + 1.\]
    \\
    \noindent Moreover, for every integer $k \geq 1$, $\iota(I_k) = k$ and $I_k$ is $(\iota, |E(I_k)|-1)$-critical, and if $k \geq 3$, then $|E(M_k)| = k$ and $M_k$ is $(\iota, k-1)$-critical.
\end{thm}

\noindent Theorem~\ref{thm-subdiv} has the following immediate consequence.

\begin{wn}\label{cor-upperq}
     If $G$ is a connected $m$-edge non-star graph, then $G$ is $(\iota,q)$-critical for some $q \leq m - 1$. 
\end{wn}

In Section~\ref{sec:1crit}, we provide a general characterization of $(\iota,1)$-critical graphs as well as a constructive characterization of $(\iota,1)$-critical trees, demonstrating that $(\iota,1)$-criticality can be checked in linear time for a tree.

\medskip
\paragraph{\textbf{Background.}} The concept of critical graphs in domination has a long history with numerous operations that alter a graph in small ways (for standard domination, see for example~\cite{AP07,BHHR05, BKK15, F96, S91,SB83,WA79, ZSZ25}).
The idea of $q$-criticality for the domination number with respect to subdivision of edges was introduced recently by Dettlaff et al.~in \cite{magda} (for additional results related to the effect of edge subdivision on the domination number of a graph, see for example~\cite{fav1,bhv,fav3,fav2}). Recall that a graph $G$ is \emph{$(\gamma,q)$-critical} if $\gamma(G_A)>\gamma(G)$ for each $A\subseteq E(G)$ with $|A|=q$, and $\gamma(G_{A'})=\gamma(G)$ for some $A'\subseteq E(G)$ with $|A'|=q-1$. The paper \cite{magda}  discusses some general properties of $(\gamma,q)$-critical graphs and characterizes the $(\gamma,|V(T)|-1)$-critical trees $T$ as well as some classes of $(\gamma,k)$-critical trees for $k \in \{2,3\}$.

\medskip
\paragraph{\textbf{Notation.}}
Let $[k]$ denote $\{1,\ldots ,k\}$. For a set $S \subseteq V(G)$, let $N_G(S)$ denote $N_G[S] \setminus S$.
An $n$-vertex path $P$ with vertex set $V(P) = \{v_1,\ldots,v_n\}$ and edge set $E(P) = \{v_1 v_2, v_2 v_3, \ldots, v_{n-1}v_n\}$ can be represented as $P = (v_1,\ldots,v_n)$. Such a path is called a $v_1$-$v_n$ path. A copy of $C_3$ is called a \emph{triangle}. 
A vertex $v \in V(G)$ is a \emph{leaf} of $G$ if $v$ has exactly one neighbour in $G$;
a vertex $u \in V(G)$ is called a \emph{support vertex} of $G$ if it is adjacent to a leaf of $G$.
For $v \in V(G)$, the \emph{degree of $v$}, denoted by $d_G(v)$, is $|N_G(v)|$.
The \emph{distance} between two vertices $x$ and $y$ of $G$, denoted by $d_G(x,y)$, is the length of the shortest $x$-$y$ path in $G$.
The \emph{diameter of $G$}, denoted by $\diam(G)$, is $\max\{d_G(u,v) \colon u,v \in V(G)\}$. If a subgraph of $G$ is a path of length $\diam(G)$, then it is called a \emph{diametral path of $G$}. 
The distance between a vertex $x$ and a set of vertices $Y$ in $G$, denoted by $d_G(x,Y)$, is $\min\{d_G(x,y) \colon y \in Y\}$. Similarly, the distance between two vertex sets $X$ and $Y$ in $G$, denoted by $d_G(X,Y)$, is $\min\{d_G(x,y) \colon x \in X, y\in Y\}$.
A subset $A$ of $V(G)$ is a $k$-\emph{packing of $G$} if $d_G(x,y)>k$ for every $x,y\in A$ with $x \neq y$.
If $G$ is a $k$-star and $k \geq 1$, then $G$ is called a \emph{non-trivial star}. 
If $G$ is a star and $N_G[v] = V(G)$, then $v$ is called the \emph{central vertex of $G$}. 
If $G$ has an $\iota(G)$-set $\{v\}$ for some $v \in V(G)$, then $v$ is called an \emph{isolating vertex of $G$}.

\section{Preliminary results}\label{sec:prelim}

The isolation number of a graph cannot decrease with the subdivision of an edge and can increase by at most one.

\medskip
\begin{prop}~\cite{DLMZZ26}
    If $G$ is a graph, then $\iota(G)\leq \iota(G_e)\leq \iota(G)+1$ for each $e \in E(G)$.
\end{prop}

\noindent We immediately obtain the following corollary.
\begin{wn}\label{cor-subdiv}
    If $G$ is a graph, then $\iota(G) \leq \iota(G_A) \leq \iota(G) + |A|$ for each $A \subseteq E(G)$.
\end{wn}

\noindent Since $\iota(G) \le n/3$ for any connected $n$-vertex graph $G$~\cite{CaHa17,Z19}, we observe that the bound $\iota(G_A) \le \iota(G)+|A|$ is tight up to $|A|=(n-3)/2$ for some graph $G$. Indeed, consider a \emph{spider} $S_t$, that is, the graph obtained from the star $K_{1,t}$ for $t \geq 2$ by subdividing each edge of the star (thus, $|V(S_t)|=2t+1$). If $A$ is a set of $t-1$ edges incident to a leaf in $S_t$, then $|A|=t-1=(|V(S_t)|-3)/2$ and \[\iota((S_t)_A)=\frac{|V({(S_t)}_A)|}{3}=t=1+\frac{|V(S_t)|-3}{2}=\iota(S_t)+|A|\] as required.

In the literature, when introducing new domination parameters, most preliminary results concern establishing tight bounds for well-known graph classes. This is not our intention. However, we need to recall two basic results that will be exploited in the proofs of the main results. Recall that from~\cite{CaHa17} we have $\iota(P_n)=\lceil (n-1)/4 \rceil$ for a path $P_n$, $n \ge 1$, while $\iota(C_n)=\lceil n/4 \rceil$ for a cycle $C_n$, $n \ge 3$. On the other hand, Dettlaff et al.~\cite{DLMZZ26} proved the following observations.

\begin{obs}\cite{DLMZZ26} 
    If $G$ is a path of order $n \ge 4$, then
    \begin{equation*}
        \sdi(G)=\begin{cases}
            1, \quad n\equiv 1\pmod 4\\
            2, \quad n\equiv 0\pmod 4\\
            3, \quad n\equiv 3\pmod 4\\
            4, \quad n\equiv 2\pmod 4.
        \end{cases}
    \end{equation*}
\end{obs}

\begin{obs}\cite{DLMZZ26} 
    If $G$ is a cycle of order $n \ge 3$, then
    \begin{equation*}
        \sdi(G)=\begin{cases}
            1, \quad n\equiv 0\pmod 4\\
            2, \quad n\equiv 3\pmod 4\\
            3, \quad n\equiv 2\pmod 4\\
            4, \quad n\equiv 1\pmod 4.
        \end{cases}
    \end{equation*}
\end{obs}

Consequently, since the subdivision of any $k$ edges of the cycle $C_n$ (the path $P_n$) {leads to a graph} isomorphic to $C_{n+k}$ ($P_{n+k}$), we immediately obtain the following observations.

\begin{obs}\label{obs-paths}
    If $n \geq 4$, then $P_n$ is $(\iota,q)$-critical with $q=\sdi(P_n)$.
\end{obs}

\begin{obs}\label{obs-cycles}
    If $n \geq 3$, then $C_n$ is $(\iota,q)$-critical with $q=\sdi(C_n)$.
\end{obs}

\section{Graphs which are \texorpdfstring{$(\iota,q)$}{(ι,q)}-critical}\label{sec:general}

We consider two different perspectives on tackling the existence of $(\iota,q)$-critical graphs, the first one stated as follows. Given an integer $q \geq 1$, can we find a connected graph which is $(\iota,q)$-critical? To answer this question, we require the following definition.

A \emph{d-wounded spider} $S_{t,t-d}$ is the graph formed by subdividing $t-d\leq t-1$ edges of a star $K_{1,t}, t \geq 2$  ($d$ is a number of edges that we do not subdivide; $d \geq 1$).

\begin{prop}\label{prop-woundedspider}
    If $T$ is a $d$-wounded spider $S_{t,t-d}$, then $T$ is $(\iota,d+1)$-critical.
\end{prop}
\proof{
    Let $S$ be a copy of $K_{1,t}$, let $x$ be the central vertex of $S$, and let $v_1,\ldots,v_t \in V(S)\setminus\{x\}$. To obtain a copy $T$ of the $d$-wounded spider $S_{t,t-d}$, simply subdivide $xv_i$ with subdivision vertex $u_i$ for every $i \in [t]\setminus[d]$. Note that $\iota(T)=1$.

    Consider now the set $A=\{xv_i \colon i \in [d]\}$. The set $\{x\}$ is an isolating set of $T$ and $T_A$. Therefore, $q\geq d$. Let $A'$ be a set of $d+1$ edges of $T$. Note that $A'$ contains at least one edge in $E(T)\setminus A$. Thus, $\diam(T_{A'}) = 5$. Hence, $\iota(T_{A'}) > 1=\iota(T)$. Therefore, $T$ is $(\iota,d+1)$-critical.\qed{}
}
\\

By Observation~\ref{obs-paths} and Proposition~\ref{prop-woundedspider}, we obtain the following corollary.

\begin{wn}\label{cor-anyq}
    There exists an $(\iota,q)$-critical tree for every $q \ge 1$.
\end{wn}

\noindent Now, we consider a different perspective. That is, given a graph $G$ which is $(\iota,q)$-critical, what are the possible values of the parameter $q$? Recall that by Corollary~\ref{cor-anyq}, we ascertain that there exist $(\iota,q)$-critical graphs with parameter $q \geq 1$. Theorem~\ref{thm-subdiv} seeks to establish an upper bound on this parameter $q$ for connected graphs which are not stars. Prior to proving Theorem~\ref{thm-subdiv}, we state two lemmas that are used frequently in studies on isolation of graphs. We also prove a separate lemma concerning the isolation number of subdivided induced subgraphs, and provide a construction of an infinite family of graphs which attain the upper bound. 

\begin{lem}[\cite{Borg1}]\label{lem-isol1}
    If $G$ is a graph, $\mathcal{F}$ is a set of graphs, $X\subseteq V(G)$, and $Y\subseteq N[X]$, then \[\iota(G,\mathcal{F}) \leq |X| + \iota(G-Y,\mathcal{F}).\]
\end{lem}

\begin{lem}[\cite{Borg1,Borgrsc}]\label{lem-isol2}
    If $G$ is a graph with components $G_1, G_2, \ldots, G_r$ and $\mathcal{F}$ is a set of connected graphs, then \[\iota(G,\mathcal{F}) = \sum_{i=1}^r\iota(G_i,\mathcal{F}).\]
\end{lem}

\begin{lem}\label{lem-ISL}
    If $G$ is an $m$-edge graph, $X \subseteq V(G)$, $A \subseteq E(G)$, and $|A| \geq m - 1$, then $\iota(G_A) \geq \iota(G[X]_A)$.
\end{lem}
\proof{
    Let $D$ be an $\iota(G_A)$-set. Let $D_X = D \cap V(G{[X]}_A)$. If $D_X$ is an isolating set of $G{[X]}_A$, then $\iota(G{[X]}_A) \leq |D_X| \leq |D| = \iota(G_A)$. Suppose $D_X$ is not an isolating set of $G{[X]}_A$. Then there exists a set $S \subseteq D \setminus V(G{[X]}_A)$ such that $D_X \cup S$ is an isolating set of $G{[X]}_A$. Let $S_X = N_{G_A}(S) \cap V(G{[X]}_A)$. Note that $S\subseteq N_{G_A}(V(G{[X]}_A)) \subseteq N_{G_A}(X)$.
    
    Suppose $S$ consists only of subdivision vertices in $V(G_A)\setminus V(G)$. Every vertex $u \in S$ has exactly one neighbour $u' \in S_X$ such that $N_{G_A}(u) \cap V(G{[X]}_A) = \{u'\} \subseteq N_{G{[X]}_A}[u']$. Therefore, $D_X \cup S_X$ is an isolating set of $G{[X]}_A$.

    Now suppose $S$ contains a vertex $v \in V(G)$ (which is not a subdivision vertex). By definition of $S$, $N_{G_A}(v) \cap V({G[X]}_A) = \{v'\} \subseteq S_X$ for some $v' \in V(G{[X]}_A)$. Note that $v v' \notin A$, so $|A| = m - 1$. Thus, $v'$ is the only neighbour of $v$ in $G[X]_A$, and $N_{G_A}(v)\setminus \{v'\} \subseteq V(G_A) \setminus (V(G) \cup V(G{[X]}_A))$. 
    Therefore, $D_X \cup S_X$ is an isolating set of $G{[X]}_A$.

    In both cases, $\iota(G{[X]}_A) \leq |D_X \cup S_X| = |D_X \cup S| \leq |D| = \iota(G_A)$.\qed{}
}

\begin{cons}\label{cons-I}\emph{
    Let $I_1$ be a $4$-vertex path $(v_{1}^{2},v_{1}^{1},v_{2}^{1},v_{3}^{1})$.
    Let $I_2$ be a $6$-vertex path $(v_{3}^{2},v_{2}^{2},v_{1}^{2},v_{1}^{1},v_{2}^{1},v_{3}^{1})$.
    For $k \geq 3$, let $B_1, B_2, \ldots, B_{k-2}$ be $3$-vertex paths such that $B_1, B_2, \ldots, B_{k-2}$ and $I_2$ are pairwise vertex-disjoint, and for $i \in [k-2]$, let $v_1^{i+2}, v_2^{i+2}, v_3^{i+2} \in V(B_i)$ such that $B_i =  (v_{1}^{i+2},v_{2}^{i+2},v_{3}^{i+2})$. For $k\geq 3$, let $I_k$ be the graph with $V(I_k) = V(I_2) \cup \bigcup_{i=1}^{k-2} V(B_i)$ and $E(I_k) = E(I_2) \cup \bigcup_{i=1}^{k-2} (E(B_i) \cup \{v_{1}^{1}v_{1}^{i+2}\})$. Let $\mathcal{I} = \{I_k \colon k \in \mathbb{N}\}$.}
\end{cons}

\begin{cons}\label{cons-M}\emph{
    Let $M_3$ be a $4$-vertex path $(w_0,w_1,w_2,w_3)$. For $k \geq 4$, let $w_4, \ldots, w_k$ be distinct vertices that are not in $V(M_3)$, and let $M_k$ be the graph with $V(M_k) = V(M_3) \cup \{w_i \colon i \in [k]\setminus[3]\}$ and $E(M_k) = E(M_3) \cup \{w_2w_i \colon i \in [k]\setminus[3]\}$. Let $\mathcal{M} = \{M_k \colon k \in \mathbb{N},k\ge 3\}$.}
\end{cons}


    \noindent\textbf{Proof of Theorem~\ref{thm-subdiv}.} If $\Delta(G) \leq 2$, then either $G\simeq C_3$ or $G$ is a $\{P_n,C_n\}$-graph for some $n \geq 4$. By Observation~\ref{obs-paths} and Observation~\ref{obs-cycles}, $G$ is $(\iota,q)$-critical where $q \leq E(G) - 1$. Thus, $\iota(G_A) \geq \iota(G) + 1$ for every $A \subseteq E(G)$ of size at least $m - 1$. Suppose $\Delta(G) \geq 3$.

    Suppose $\iota(G) = 1$. 
    If $G$ contains a triangle $T$, then $T_A$ is a $\{C_5,C_6\}$-graph, so $\iota(G_A) \geq \iota(T_A) = 2 = \iota(G) + 1$ by Lemma~\ref{lem-ISL}. Suppose that $G$ contains no triangle.
    Let $v$ be an isolating vertex of $G$ of largest degree. Then, $d_G(v) \geq 2$, because if we assume that $v$ is a leaf of $G$, then we obtain that its neighbour is an isolating vertex of $G$ whose degree is larger than $d_G(v)$. Since $G$ is a connected non-star graph containing no triangle, $ux \in E(G)$ for some $x \in N_G(v)$ and $u \in V(G)\setminus N[v]$. Since $d_G(v) \geq 2$, $G$ has a vertex $y$ in $N_G(v)\setminus\{x\}$. Since $G$ contains no triangle, $xy \notin E(G)$. If $uy \notin E(G)$, then $G{[\{v,x,y,u\}]}_A$ is a $\{P_6,P_7\}$-graph. If $uy \in E(G)$, then $G{[\{v,x,y,u\}]}_A$ is a $\{C_7,C_8\}$-graph. In any case, by Lemma~\ref{lem-ISL}, $\iota(G_A) \geq \iota(G{[\{v,x,y,u\}]}_A) = 2 = \iota(G) + 1$.

    Now suppose $\iota(G) \geq 2$. We proceed by induction on $n$.
    If $|A| = m - 1$, then let $e$ be the unique member of $E(G)\setminus A$. If $|A| = m$, then let $e \in E(G)$. Since $G$ is connected, $d_G(v) \geq 2$ for some $v \in e$. Let $G' = G-N_G[v]$ and let $S$ be the set of subdivision vertices of edges $xy$ such that $x \in N_G(v)$ and $y \in V(G')$. Note that every vertex $z$ in $S$ has exactly one neighbour $y$ in $V(G'_A)$. Let $D$ be an $\iota(G_A)$-set such that $D \cap S$ is of minimum size. Let $S' = N_{G_A}(D \cap S) \cap V(G'_A)$ and let $D' = (D\setminus S) \cup S'$. Since $N_{G_A}(D \cap S) \setminus N_G(v) = S'$, we note that $D'$ is an isolating set of $G'_A$. Since $d_G(v) \geq 2$ and $|A| \geq m - 1$, $|D \cap V(G{[N_G[v]]}_A)| \geq 1$. Therefore, by Corollary~\ref{cor-subdiv}, 
    \begin{equation}
        \iota(G_A) = |D| \geq |D'| + 1 \geq \iota(G'_A) + 1 \geq \iota(G') + 1. \label{eq-outside}
    \end{equation}
    If some inequality in~(\ref{eq-outside}) is strict, then by Lemma~\ref{lem-isol1} (with $X = \{v\}$ and $Y = N_G[v]$), $\iota(G_A) \geq \iota(G') + 2 \geq \iota(G) + 1$. Suppose equality in~(\ref{eq-outside}) holds throughout. Let $H_1, \ldots, H_r$ be the components of $G'$ such that $\iota(H_1) \geq \iota(H_2) \geq \cdots \geq \iota(H_r)$. By Lemma~\ref{lem-isol2}, $\iota(G'_A) = \sum_{i=1}^{r}\iota({(H_i)}_A)$. If for some $j \in [r]$, $H_j$ is not a star, then by the induction hypothesis, $\iota({(H_j)}_A) \geq \iota(H_j) + 1$. Thus, $\iota(G'_A) \geq \sum_{i=1}^{r}\iota(H_i) + 1 = \iota(G') + 1$, contradicting $\iota(G'_A)=\iota(G')$. Thus, every component of $G'$ is a star. Let $s = \max\{i \in [r]: \iota(H_i) \neq 0\}$. Recall from equality in~(\ref{eq-outside}) that $D'$ is an $\iota(G'_A)$-set. Since $|D'| = \iota(G'_A) = \sum_{i=1}^{r}\iota(H_i) = s$, we have $D' = \{u_1,\ldots,u_s\}$ where $u_i \in V({(H_i)}_A)$ for each $i \in [s]$.

    We claim $D \cap S = \emptyset$. Suppose for contradiction that $D\cap S \neq \emptyset$. Let $z \in D \cap S$. Since $D'$ consists only of vertices from components of $G'$ which are non-trivial stars, $z \in N_{G_A}(V(H_j))$ for some $j \in [s]$. By definition of $D'$, $z$ is the subdivision vertex of the edge $x_j u_j$ for some $x_j \in N_G(v)$ and some $u_j \in V(H_j)$. Note that $D' \cap V(H_j) = \{u_j\}$. Since $H_j$ is not a trivial star, $u_j y_j \in E(G)$ for some $y_j \in V(H_j)\setminus\{u_j\}$. Let $y'$ be the subdivision vertex of $u_j y_j$. Suppose $u_j \notin D$. Note that $N_{G_A}(y_j) \subseteq V(G_A)\setminus V(G)$. Since $z \in D$, $u_j \notin D$, and $|D' \cap V({(H_j)}_A)| = 1$, we have $D \cap V({(H_j)}_A) = \emptyset$. If a vertex of $N_{G_A}(y_j) \cap S$ is in $D$, then $y_j \in D'$, contradicting $|D' \cap V({(H_j)}_A)| \leq 1$. Thus, $D \cap N_{G_A}[y_j] = \emptyset$. Note that $N_{H_j}[y_j] \cap N_{G_A}[z] = \emptyset$, so $y_j y' \in E(G_A-N_{G_A}[D])$, contradicting that $D$ is an isolating set of $G_A$. 
    Thus, $u_j \in D$. Recall that $z \in D$ and $z \in N_{G_A}(u_j)$. Let $D^* = (D \setminus \{z\}) \cup \{x_j\}$. Note that $D^*$ is an $\iota(G_A)$-set such that $|D^* \cap S| < |D \cap S|$, contradicting the choice of $D$ with $D \cap S$ being of minimum size. Therefore, $D \cap S = \emptyset$.

    Suppose $u_j \notin V(H_j)$ for some $j \in [s]$. Thus, $u_j$ is a subdivision vertex in $V({(H_j)}_A)$. Since $\{u_j\}$ is an isolating set of ${(H_j)}_A$, we have $H_j \simeq K_2$. Let $V(H_j) = \{y_1,y_2\}$. Observe that $N_{G_A}(u_j) = \{y_1,y_2\}$ and ${(G-\{y_1,y_2\})}_A = {(G-V(H_j))}_A$. Suppose $G-V(H_j)$ is not a star. By the induction hypothesis, we have $|D \setminus \{u_j\}|=|D|-1 \geq \iota({(G-V(H_j))}_A) \geq \iota(G-V(H_j)) + 1$. Therefore, by Lemma~\ref{lem-isol1} (with $X = \{y_1\}$ and $Y = \{y_1,y_2\} = V(H_j)$), $\iota(G_A) = |D| \geq \iota(G-V(H_j)) + 2 \geq \iota(G) + 1$. Now suppose $G-V(H_j)$ is a star. Thus, $V(G) = N[v] \cup \{y_1,y_2\}$ and $\iota(G) = 2$. Since $G$ is connected, $x_1y_1 \in E(G)$ for some $x_1 \in N_G(v)$. If $N_G(v) \cap N_G(\{y_1,y_2\}) = \{x_1\}$, then $\iota(G) = |\{x_1\}| = 1$, a contradiction. Let $x_2 \in (N_G(v)\cap N_G(\{y_1,y_2\})) \setminus \{x_1\}$. If $N_G(y_i) = \{x_1,x_2\}$ for some $i \in [2]$, then $\iota(G) = 1$, a contradiction. Thus, $x_1y_1, x_2y_2 \in E(G)$ and $x_1y_2,x_2y_1 \notin E(G)$. Since $G-V(H_j)$ is a star, $x_1x_2 \notin E(G)$. Hence, $G[\{v,x_1,x_2,y_1,y_2\}] \simeq C_5$ and ${G[\{v,x_1,x_2,y_1,y_2\}]}_A$ is a $\{C_9,C_{10}\}$-graph. By Lemma~\ref{lem-ISL}, $\iota(G_A) \geq \iota({G[\{v,x_1,x_2,y_1,y_2\}]}_A) = 3 \geq \iota(G) + 1$. 
    
    Now suppose $u_i \in V(H_i)$ for every $i \in [s]$, i.e. no $u_i$ is a subdivision vertex. Since $D'$ is an $\iota(G'_A)$-set and $D' \cap V({(H_i)}_A) = \{u_i\}$ for every $i \in [s]$, we have $N_{H_i}[u_i] = V(H_i)$ for every $i \in [s]$. 
    Let $G^* = G-u_1$ and let $G_v^*$ be the component of $G^*$ such that $v \in V(G_v^*)$. Note that the components of $G^*$ consist of $G_v^*$ and possibly some isolated vertices in $V(H_1)\setminus \{u_1\}$. Since $N_{G_A}(u_1) \cap V({(G_v^*)}_A) = \emptyset$, $D \setminus \{u_1\}$ is an isolating set of ${(G_v^*)}_A$. Thus, $|D \setminus \{u_1\}| \geq \iota({(G_v^*)}_A)$. If $G_v^*$ is not a star, then by the induction hypothesis, $\iota({(G_v^*)}_A) \geq \iota(G_v^*) + 1$. Thus, by Lemma~\ref{lem-isol1} (with $X = Y = \{u_1\}$) and Lemma~\ref{lem-isol2}, $|D| \geq \iota(G_v^*) + 2 = \iota(G^*) + 2 \geq \iota(G) + 1$. If $G_v^*$ is a star, then $V(G_v^*) = N[v]$ and $G' = H_1$ (so $r=1$). Since $G$ is connected, $N_G(v) \cap N_G(V(H_1)) = \{u_1\}$. Let $x\in N_G(v) \cap N_G(u_1)$. Note that $x$ is an isolating vertex of $G$, contradicting $\iota(G) \geq 2$. Thus, the first part of the theorem is proved.
    
    We now prove the second part of the theorem. Consider the graph family $\mathcal{I}$ defined in Construction~\ref{cons-I} (an example is illustrated in Figure~\ref{fig:n3}). Observe that for every $k \in \mathbb{N}$, $\{v_1^i \colon i \in [k]\}$ is an $\iota(I_k)$-set, so $\iota(I_k) = k$. For every $k \in \mathbb{N}$, let $A_k^\iota = E(I_k) \setminus \{{v_1^1}{v_2^1},{v_2^1}{v_3^1}\}$. Note that $\iota(I_k) = \iota({(I_k)}_{A_k^\iota})$, so $I_k$ is $(\iota,q)$-critical for some $q > |E(I_k)|-2$. By the first part of the theorem statement, 
    $q = |E(I_k)| - 1$.  

    Now consider the graph family $\mathcal{M}$ defined in Construction~\ref{cons-M}. Observe that for every $k \geq 3$, $|E(M_k)| = k$.
    For every $k \geq 3$, let $A_k = E(M_k) \setminus \{w_0w_1,w_1w_2\}$. Note that $\iota(M_k) = \iota({(M_k)}_{A_k}) = |\{w_2\}|$, so $M_k$ is $(\iota,q)$-critical for some $q > k-2$. By the first part of the theorem statement, 
    $q = k-1$.\qed{}

\medskip

\vfill

\begin{figure}[ht]
    \centering
    \begin{tikzpicture}[scale=1.0]
        \node [v] (x1) at (0,0) {};
        \node [v] (x2) at (0,1) {};
        \node [v] (x3) at (0,2) {};
    
        \node [v] (y1) at (1,0) {};
        \node [v] (y2) at (1,1) {};
        \node [v] (y3) at (1,2) {};
    
        \node [v] (z1) at (2,0) {};
        \node [v] (z2) at (2,1) {};
        \node [v] (z3) at (2,2) {};
    
        \node [v] (v) at (3,1) {};
        \node [v] (x) at (4,1) {};
        \node [v] (y) at (5,1) {};
    
        \draw[e] (x1) -- (y1) -- (z1) -- (v);
        \draw[e] (x2) -- (y2) -- (z2) -- (v);
        \draw[e] (x3) -- (y3) -- (z3) -- (v);
        \draw[e] (v) -- (x) -- (y);
    
        \begin{scope}[on background layer]
            \draw[highlight] (x1) -- (y1);
            \draw[highlight] (y1) -- (z1);
            \draw[highlight] (z1) -- (v);
            \draw[highlight] (x2) -- (y2);
            \draw[highlight] (y2) -- (z2);
            \draw[highlight] (z2) -- (v);
            \draw[highlight] (x3) -- (y3);
            \draw[highlight] (y3) -- (z3);
            \draw[highlight] (z3) -- (v);
            \foreach \n in {x1, y1, z1, x2, y2, z2, x3, y3, z3, v} {
                \fill[gray!50] (\n) circle (4pt);
            }
        \end{scope}
    \end{tikzpicture}
    \caption{The graph $I_4$ with the edges of $A_4^\iota$ highlighted in grey. Note that $\iota(I_4) = \iota({(I_4)}_{A_4^\iota})$.}\label{fig:n3}
\end{figure}

\vfill

\section{\texorpdfstring{$(\iota,1)$}{(ι,1)}-critical graphs}\label{sec:1crit}

In this section, we focus on various characterizations of $(\iota,1)$-critical graphs. 

\subsection{The general case} 

The first characterization follows a similar pattern for the characterization of $(\gamma,1)$-critical graphs in~\cite{rad-domcrit}, given below.

\begin{thm}\label{thm-domcrit}\cite{rad-domcrit}
    A graph $G$ is $(\gamma,1)$-critical if and only if every $\gamma(G)$-set is a 2-packing.
\end{thm}

This approach for $(\iota,1)$-critical graphs requires another layer of intricacy to encapsulate a similar characterization due to the added difficulty imposed by the isolated vertices obtained from the deletion of the closed neighbourhood of a minimum isolating set of an $(\iota,1)$-critical graph. We begin by providing a definition which is the crux of this approach.

\begin{df}\label{def-ABC} 
    If $G$ is a graph, $A$, $B$ and $C$ are non-empty sets that partition $V(G)$,
    \begin{enumerate}[label=(\roman*), align=left, leftmargin=*]
        \item $A \cup C$ and $B$ are independent sets of $G$,
        \item $N(A) = B = N(C)$,
        \item $A$ is a $3$-packing of $G$, and
        \item no vertex in $A$ is a leaf of $G$,
    \end{enumerate}
    then we say that $(A, B, C)$ is a \emph{critical tripartition of $G$}.
\end{df}

Note that for a graph $G$ with a critical tripartition $(A,B,C)$, $A\cup C$ and $B$ are two independent sets such that $(A \cup C) \cup B = V(G)$ and $(A \cup C) \cap B = \emptyset$. Therefore, $G$ is a bipartite graph with partite sets $A \cup C$ and $B$. We note the following observation arising out of Definition~\ref{def-ABC} and the property that a graph is bipartite if and only if it has no odd cycle~\cite{konig}.

\newpage
\begin{obs}\label{obs-ABC}
    If $G$ is a graph that has a critical tripartition $(A,B,C)$, then
    \begin{enumerate}[label=(\roman*), align=left, leftmargin=*]
    \item no vertex in $A\cup C$ is a support vertex of $G$,
    \item $G$ contains no odd cycle,
    \item the distance between any two vertices in $A \cup C$ is even, and
    \item every leaf of $G$ is in $C$.
    \end{enumerate}
\end{obs}

\begin{thm}\label{thm-ABC}
    A connected graph $G$ is $(\iota,1)$-critical if and only if $(D,N(D),V(G)\setminus N[D])$ is a critical tripartition of $G$ for each $\iota(G)$-set $D$.
\end{thm}
\proof{
    Let $G$ be a connected $(\iota,1)$-critical graph. Let $D$ be an $\iota(G)$-set.
    Let $R = V(G) \setminus N[D]$. Then, no vertex in $D$ is adjacent to a vertex in $R$. By the choice of $D$, $R$ is an independent set of $G$.

    \begin{claim}\label{claim-3p} $D$ is a $3$-packing of $G$.
    \end{claim}
    \proof{
        Suppose that $D$ is not a $3$-packing of $G$. Then, $D$ has two distinct vertices $x$ and $y$ with $d(x,y) \leq 3$. Let $P$ be a shortest path from $x$ to $y$ in $G$. Since $d(x,y) \leq 3$, $|V(P)| \leq 4$ and $V(P) = N_P[\{x,y\}]$. Let $e \in E(P)$ and consider the graph $G_e$. 
        We have $d_{G_e}(x,y) \leq 4$, $|V(P_e)| \leq 5$ and $|V(P_e) \setminus N_{P_e}[\{x,y\}]| \leq 1$. Thus, $D$ is an $\iota(G_e)$-set. This contradicts the assumption that $G$ is $(\iota,1)$-critical. \qed{}
    }\medskip

    \noindent By Claim~\ref{claim-3p} and the definition of $R$, $D \cup R$ is an independent set of $G$.

    \begin{claim}\label{claim-NDind}
        $N_G(D)$ is an independent set of $G$.
    \end{claim}
    \proof{
        Suppose $uv \in E(G)$ for some $u,v \in N_G(D)$. Let $x$ be the subdivision vertex of $uv$. Note that $N_{G_{uv}}(x) = \{u,v\} \subseteq N_G(D)$. Hence, $N_{G_{uv}}[x]\setminus N_{G_{uv}}[D] =\{x\}$ 
        and $V(G_{uv})\setminus N_{G_{uv}}[D] = R\cup\{x\}$. Thus, $\iota(G_{uv}) \leq |D| = \iota(G)$, contradicting $G$ being $(\iota,1)$-critical. Therefore, $N_G(D)$ is an independent set of $G$.\qed{}
    }

    \begin{claim}\label{claim-nxr}
        For every $x \in N_G(D)$, $N_G(x) \cap R \neq \emptyset$.
    \end{claim}
    \proof{
        Suppose for contradiction that there exists some $x \in N_G(D)$ such that $N_G(x) \cap R = \emptyset$. Let $w \in N_G(x) \cap D$. By Claim~\ref{claim-NDind}, $N_G(x) \subseteq D$. Thus, any edges of $G_{wx}$ not in $E(G)$ have a vertex in $N_{G_{wx}}[D]$. Hence, $\iota(G_{wx}) \leq |D| = \iota(G)$. This contradicts $G$ being $(\iota,1)$-critical.\qed{}
    }\medskip

    \noindent By Claim~\ref{claim-nxr}, $N_G(R) = N_G(D)$.

    \begin{claim}\label{claim-xy4}
        If $|D| \geq 2$, then for every $x \in D$, there exists some $y \in D\setminus\{x\}$ such that $d_G(x,y) = 4$.
    \end{claim}
    \proof{
        Suppose for contradiction that there exists some $x \in D$ such that $d_G(x,y) \geq 5$ for every $y \in D\setminus\{x\}$. Let $P$ be a shortest path from $x$ to some $x'\in D\setminus\{x\}$ in $G$. Let $P^* = P - N_P[\{x,x'\}]$. Since $d_G(x,x') \geq 5$, $|V(P)| \geq 6$ and $P^* \simeq P_k$ for some $k \geq 2$. Let $v,w \in V(P^*)$ such that $d_G(x,v) = 2$ and $d_G(x,w) = 3$. Thus, $vw \in E(G)$. Since $D$ is an $\iota(G)$-set, $N_G[D] \cap \{v,w\} \neq \emptyset$. However, by Claim~\ref{claim-3p}, $D$ is a $3$-packing in $G$, and hence $D \cap N_G[v] = \emptyset$. Thus, $D \cap (N_G(w)\setminus N_G[v]) \neq \emptyset$. Note that $d_G(x,z) = 4$ for each $z \in D \cap (N_G(w)\setminus N_G[v])$, which contradicts that $d_G(x,y) \geq 5$ for every $y \in D\setminus\{x\}$.\qed{}
    }

    \begin{claim}
        For every $x \in D$, $d(x) \geq 2$.
    \end{claim}
    \proof{
        Suppose $d_G(x) = 1$ for some $x \in D$. If $|D| = 1$, then $G \simeq K_{1,k}$ for some $k \in \mathbb{N}$, which contradicts that $G$ is $(\iota,1)$-critical. Suppose $|D|\geq 2$. By Claim~\ref{claim-xy4}, there exists some $y \in D\setminus\{x\}$ such that $d_G(x,y) = 4$. Let $w \in N_G(x)$.
        Since $N_G[x] \subseteq N_G[w]$, $(D\setminus\{x\}) \cup \{w\}$ is also an $\iota(G)$-set with $d_G(w,y) = 3$, which contradicts that every $\iota(G)$-set is a $3$-packing in $G$ (Claim~\ref{claim-3p}).\qed{}
    }\medskip

    In conclusion, $D \cup N_G(D) \cup R = V(G)$, $D \cup R$ and $N_G(D)$ are independent sets of $G$, $N_G(D) = N_G(R)$, $D$ is a $3$-packing in $G$, and no vertex in $D$ is a leaf. Therefore, $(D,N_G(D),R)$ is a critical tripartition of $G$.\medskip

    We now prove the converse. Let $G$ be a connected graph such that $(D,N(D),\allowbreak V(G)\setminus N[D])$ is a critical tripartition of $G$ for each $\iota(G)$-set $D$. Suppose that $G$ is not $(\iota,1)$-critical. Then, $\iota(G_{xy}) = \iota(G)$ for some $xy \in E(G)$. Let $D'$ be an $\iota(G_{xy})$-set and let $z \in V(G_{xy})\setminus V(G)$ be the subdivision vertex of $xy$. Note that $|D'| = \iota(G)$.

    Suppose $z \notin D'$. Since $D' \subseteq V(G)$, $D'$ is also an $\iota(G)$-set. Let $R = V(G)\setminus N_G[D']$. Thus, $(D',N_G(D'),R)$ is a critical tripartition of $G$. Since $D'$ is an isolating set of $G_{xy}$, $xz,yz \notin E(G_{xy} - N_{G_{xy}}[D'])$, so $N_G[D'] \cap \{x,y\} \neq \emptyset$.
    
    Suppose $d_G(x)=1$. By Definition~\ref{def-ABC}~(iv) and Observation~\ref{obs-ABC}~(i), $x \in R$, $y \in N_G(D')$, and $N_G(y)\setminus\{x\} \subseteq D'$. This contradicts $D'$ being an $\iota(G_{xy})$-set since $xz \in E(G_{xy}-N_{G_{xy}}[D'])$. A similar argument holds for $d_G(y)=1$. 
    
    Therefore, $d_G(x) \geq 2$ and $d_G(y) \geq 2$. 
    If $x,y \notin D'$, then $x',y' \in D'$ for some $x' \in N_G(x)\setminus{y}$ and some $y' \in N_G(y)\setminus\{x\}$. Thus, $x,y \in N_G(D')$. However $xy \in E(G)$, which contradicts $N_G(D')$ being an independent set of $G$. Therefore, $D' \cap \{x,y\} \neq \emptyset$. By Definition~\ref{def-ABC}~(iii), $D'$ is a $3$-packing in $G$. Thus, either $x \in D'$ or $y \in D'$. 
    Without loss of generality, suppose $x \in D'$. Since $D'$ is a $3$-packing in $G$, $D' \cap N_G[y'] = \emptyset$ for every $y' \in N_G(y)\setminus\{x\}$. By Observation~\ref{obs-ABC}~(ii), $E(G[N_G(y)])= \emptyset$. Thus, $(N_{G_{xy}}[y]\setminus\{z\}) \cap N_{G_{xy}}[D'] = \emptyset$. But $E(G_{xy}[N_{G_{xy}}[y]\setminus\{z\}]) \neq \emptyset$, contradicting $D'$ being an isolating set of $G_{xy}$.

    Therefore, $z \in D'$. Let $D_x = (D' \setminus \{z\}) \cup \{x\}$ and let $D_y = (D' \setminus \{z\}) \cup \{y\}$.  Since $N_{G_{xy}}(z) = \{x,y\} \subseteq N_G[x] \cap N_G[y]$ and $\iota(G_{xy}) = \iota(G)$, both $D_x$ and $D_y$ are $\iota(G)$-sets. If $x \in D'$, then $|D_x| = |D' \setminus \{z\}| < |D'|$, contradicting $\iota(G) = \iota(G_{xy})$. 
    A similar argument holds for $y \in D'$. Therefore, $x,y \notin D'$.
    Let $R_x = V(G)\setminus N_G[D_x]$. Thus, $(D_x, N_G(D_x),R_x)$ is a critical tripartition of $G$. 
    Note that $D'\setminus\{z\} \subseteq D_x$. By Definition~\ref{def-ABC}~(iii), $D_x$ is a $3$-packing in $G$. By Definition~\ref{def-ABC}~(iv) and Observation~\ref{obs-ABC}~(i), neither $x$ nor $y$ is a leaf or a support vertex of $G$.
    Let $y' \in N_G(x)\setminus\{y\}$. Note that $y,y' \in N_G(D_x)$ 
    and thus by Definition~\ref{def-ABC}~(ii), $N_G(y)\cap R_x \neq \emptyset \neq N_G(y') \cap R_x$. By Definition~\ref{def-ABC}~(i), $N_G(D_x)$ is independent, so $yy' \notin E(G)$. 
    Since $D_x$ is a $3$-packing in $G$, $D_x \cap N_G[y] = D_x \cap N_G[y'] = \{x\}$. 
    Thus, $N_G[(D'\setminus\{z\})] \cap N_G[y] = \emptyset = N_G[(D'\setminus\{z\})] \cap N_G[y']$. Since $N_{G_{xy}}[z] = \{x,z,y\}$ and $yy' \notin E(G)$, $V(G_{xy}-N_{G_{xy}}[D']) = R_x \cup \{y'\}$. Furthermore, since $N_{G_{xy}}(y') \cap R_x = N_G(y') \cap R_x \neq \emptyset$, $E(G_{xy}-N_{G_{xy}}[D']) = E(G_{xy}[R_x \cup \{y'\}]) \neq \emptyset$.
    This contradicts $D$ being an $\iota(G_{xy})$-set. Therefore, $G$ is $(\iota,1)$-critical.\qed{}

}
\setcounter{claim}{0}

\subsection{The case for trees} 

Now we provide a constructive characterization of $(\iota,1)$-critical trees. 
Let $\Fi$ be the family of trees such that $P_5 \in \Fi$ and any tree $T \in \Fi \setminus \{P_5\}$ can be constructed from $P_5$ by using operations $\mathcal{O}_1$, $\mathcal{O}_2$ and $\mathcal{O}_3$ defined and illustrated below.
For every $T \in \Fi$, we define the status of a vertex $v \in V(T)$, denoted by $\sta(v)$, to be $A, B$ or $C$ in the following manner. Initially, for $P_5$, let $\sta(v) = A$ if $v$ is neither a leaf of $P_5$ nor a support vertex of $P_5$, let $\sta(v) = B$ if $v$ is a support vertex of $P_5$ and let $\sta(v) = C$ if $v$ is a leaf of $P_5$. Once a vertex is assigned a status, this status remains unchanged as the tree is recursively constructed.

\begin{itemize}
    \item Operation $\mathcal{O}_1$: The tree $T$ is obtained from some $T' \in \Fi\setminus\{T\}$ of smaller order by adding a vertex $z$ and an edge $yz$, where $y$ is a vertex of $T'$ with status $B$. Let $\sta(z)=C$.
    
    \begin{figure}[htbp]
        \centering
        \begin{tikzpicture}[scale=1.0]
            \node[v,label={north:$A$}] (x) at (-1,0) {};
            \node[v,label={north:$B$}] (y) at (0,0) {};
            \node[label={south:$y$}] at (0,0) {};
            \node[v,label={north:$C$}] (z) at (1.25,0) {};
            \node[label={south:$z$}] at (1.25,0) {};
            \draw[e] (x) -- (y) -- (z);
            \draw[e] (y) to (0.5,-0.25);
            \draw[e] (x) to (-0.5,-0.25);
            \draw[rounded corners,densely dashed,thick] (-1.5, -0.6) rectangle (0.6, 0.6);
        \end{tikzpicture}
        \caption{Operation $\mathcal{O}_1$.}
        \label{fig:o1}
    \end{figure}
    
    \item Operation $\mathcal{O}_2$: The tree $T$ is obtained from some $T' \in \Fi\setminus\{T\}$ of smaller order by adding a path $P_2=(y,z)$ and an edge $xy$, where $x$ is a vertex of $T'$ with status $A$. Let $\sta(y)= B$ and $\sta(z)=C$.
    
    \begin{figure}[htbp]
        \centering
        \begin{tikzpicture}[scale=1.0]
            \node[v,label={north:$A$}] (x) at (0,0) {};
            \node[label={south:$x$}] at (0,0) {};
            \node[v,label={north:$B$}] (y) at (1.25,0) {};
            \node[label={south:$y$}] at (1.25,0) {};
            \node[v,label={north:$C$}] (z) at (2.25,0) {};
            \node[label={south:$z$}] at (2.25,0) {};
            \draw[e] (x) -- (y) -- (z);
            \draw[e] (x) to (0.5,-0.25);
            \draw[rounded corners,densely dashed,thick] (-0.5, -0.6) rectangle (0.6, 0.6);
        \end{tikzpicture}
        \caption{Operation $\mathcal{O}_2$.}
        \label{fig:o2}
    \end{figure}
    
    \item Operation $\mathcal{O}_3$: The tree $T$ is obtained from some $T' \in \Fi\setminus\{T\}$ of smaller order by adding a path $P_4=(w,x,y,z)$ and an edge $vw$, where $v$ is a vertex of $T'$ with status $C$. Let $\sta(w)= B$, $\sta(x)= A$, $\sta(y)= B$, and $\sta(z)=C$.
    \begin{figure}[htbp]
        \centering
        \begin{tikzpicture}[scale=1.0]
            \node[v,label={north:$A$}] (t) at (-1,0) {};
            \node[v,label={north:$B$}] (u) at (0,0) {};
            \node[v,label={north:$C$}] (v) at (1,0) {};
            \node[label={south:$v$}] at (1,0) {};
            \node[v,label={north:$B$}] (w) at (0,-1.25) {};
            \node[label={south:$w$}] at (0,-1.25) {};
            \node[v,label={north:$A$}] (x) at (-1,-1.875) {};
            \node[label={south:$x$}] at (-1,-1.875) {};
            \node[v,label={north:$B$}] (y) at (0,-2.5) {};
            \node[label={south:$y$}] at (0,-2.5) {};
            \node[v,label={north:$C$}] (z) at (1,-2.5) {};
            \node[label={south:$z$}] at (1,-2.5) {};
            \draw[e] (t) -- (u) -- (v) -- (w) -- (x) -- (y) -- (z);
            \draw[e] (t) to (-0.5,-0.25);
            \draw[rounded corners,densely dashed,thick] (-1.5, -0.6) rectangle (1.5, 0.6);
        \end{tikzpicture}
        \caption{Operation $\mathcal{O}_3$.}
        \label{fig:o3}
    \end{figure}

\end{itemize}

The following observations follow directly from the construction of the family $\mathcal{F}_\iota$. Denote by $\mathcal{A}_T, \mathcal{B}_T, \mathcal{C}_T$ the sets of vertices of $T$ with status $A$, $B$, and $C$, respectively.

\begin{obs}\label{obs-fiota}
    If $T \in \mathcal{F}_\iota$, then:
    \begin{enumerate}[label=(\roman*), align=left, leftmargin=*]
        \item every leaf is in $\mathcal{C}_T$;
        \item every support vertex is in $\mathcal{B}_T$;
        \item the set $\mathcal{A}_T$ is a unique isolating set which is a $3$-packing of $T$;
        \item  $\mathcal{A}_T\cup \mathcal{C}_T$ and $ \mathcal{B}_T$ are independent sets;
        \item the distance between any two leaves is even.
        \item $N_T(\mathcal{A}_T)=\mathcal{B}_T=N_T(\mathcal{C}_T)$.
    \end{enumerate}
\end{obs}

Note that $\mathcal{F}_\iota$ is a subfamily of the family of trees which have a unique smallest vertex-edge dominating set characterized in~\cite{sentUVED}, with the first graph $P_5$ and $\mathcal{O}_1$, $\mathcal{O}_2$, and $\mathcal{O}_3$ being the same for both families of trees. Thus, Observation~\ref{obs-fiota}~(iii) is a consequence of Theorem~9 in~\cite{sentUVED}. 

Given this condition, we now show that a graph belongs to $\Fi$ if and only if it is an $(\iota,1)$-critical tree. This is motivated by the additional condition regarding the uniqueness of a minimum isolating set, which is not present in the characterization provided by Theorem~\ref{thm-ABC}. The smallest example of a non-tree graph which demonstrates this is $C_4$, which is $(\iota,1)$-critical by Observation~\ref{obs-cycles}, and we note that every vertex of $C_4$ is an isolating vertex of $C_4$. Prior to proving the necessary and sufficient conditions for the equivalence, we give the following observation for use as a base case and a proposition which is a corollary of Theorem~9 in~\cite{sentUVED}.

\begin{obs}
    If $T$ is an $(\iota,1)$-critical tree with $|V(T)| \leq 5$, then $T \simeq P_5$.
    \label{obs:P5}
\end{obs}

\begin{prop}\cite{sentUVED}\label{prop-unique}
    If $T \in \mathcal{F}_\iota$, then $T$ has a unique $\iota(T)$-set.
\end{prop}

\begin{thm}\label{thm-i1cons}
    Let $T$ be a tree. If $T$ is $(\iota,1)$-critical, then $T \in \mathcal{F}_\iota$.
\end{thm}
\proof{
    Let $T$ be an $n$-vertex $(\iota,1)$-critical tree and let $D$ be an $\iota(T)$-set. By Theorem~\ref{thm-ABC}, $(D,N_T(D),V(T)\setminus N_T[D])$ is a critical tripartition of $T$. We show that if $T$ is $(\iota,1)$-critical, then it can be constructed from $P_5$ using a sequence of operations $\mathcal{O}_1$, $\mathcal{O}_2$, or $\mathcal{O}_3$. We prove this using induction on $n$.

    Let $\diam(T) = d$. By Observation~\ref{obs:P5}, $n \geq 5$.
    Since $P_5\in \mathcal{F}_{\iota}$, we assume that $n \geq 6$.
    Let $P=(v_0,v_1,\dots ,v_d)$ be a diametral path of $T$. Observe that if $d\leq 3$, then $T$ is not $(\iota,1)$-critical. Hence, $d\geq 4$. Since $T$ is a tree, $d_T(v_0) = d_T(v_d) = 1$. By Observation~\ref{obs-ABC}~(i), $v_0,v_d \in V(T) \setminus N_T[D]$.

    Suppose $d_T(v_1) \geq 3$. Then there exists $x \in N_T(v_1)\setminus \{v_0,v_2\}$ such that $d_T(x) = 1$.
    Let $T' = T-\{x\}$. Note that $D$ is an isolating set of $T'$. Suppose for contradiction that there exists an $\iota(T')$-set $D'$ with $|D'| < |D|$. Since $v_1$ is a support vertex of $T'$ and $D'$ is a minimum isolating set of $T'$ whose closed neighbourhood intersects $v_0 v_1$, $|N_{T'}[v_1] \cap D'| = 1$. Thus, without loss of generality, suppose $v_2 \in D'$. Hence, $D'$ is an isolating set of $T$. This gives $\iota(T) \leq |D'| < |D|$, a contradiction. Thus, $\iota(T') = \iota(T)$ and $(D,N_{T'}(D),V(T')\setminus N_{T'}[D])$ is a critical tripartition of $T'$ for every $\iota(T)$-set $D$. By Theorem~\ref{thm-ABC}, $T'$ is $(\iota,1)$-critical. Thus, by the induction hypothesis, $T' \in \Fi$. By Observation~\ref{obs-fiota}~(ii), $v_1$ has status $B$, and $T$ can be obtained by applying $\mathcal{O}_1$ to $T'$ on $v_1$. Therefore, $T \in \Fi$.
    
    Now suppose $d_T(v_1) = 2$. Suppose $d_T(v_2) \geq 3$. By Observation~\ref{obs-ABC}~(iii) and (iv), $v_2$ is not a support vertex of $T$, so there exists a leaf $x \in N_T(y)$ for some $y\in N_T(v_2)\setminus \{v_1,v_3\}$. If $d_T(y) \geq 3$, then the case follows a similar argument to the case for $d_T(v_1) \geq 3$, so suppose $d_T(y) = 2$. Let $T' = T - \{x,y\}$. Note that $D$ is an isolating set of $T'$. Using similar arguments as in the previous case we can deduce that $D$ is an $\iota(T')$-set. Hence, $\iota(T) = \iota(T')$ and $(D,N_{T'}(D),V(T')\setminus N_{T'}[D])$ is a critical tripartition of $T'$ for every $\iota(T)$-set $D$. By Theorem~\ref{thm-ABC}, $T'$ is $(\iota,1)$-critical. Thus, by the induction hypothesis, $T' \in \Fi$. Since $v_2$ is neither a leaf nor a support vertex of $T'$, then by Observation~\ref{obs-fiota}~(i) and (ii), $v_2$ has status $A$, and $T$ can be obtained by applying $\mathcal{O}_2$ to $T'$ on $v_2$. Therefore, $T \in \Fi$.

    Hence, $d_T(v_2) = 2$. Since $v_2$ is not a leaf or a support vertex, then by Observation~\ref{obs-ABC}~(i) and (iv), $v_2 \in D$ for every $\iota(T)$-set $D$. Thus, $v_3 \in N_T(D)$ and $N_T(v_3)\setminus \{v_2\} \subseteq V(T) \setminus N_T[D]$ for every $\iota(T)$-set $D$. Suppose $d_T(v_3) \geq 3$. 
    Thus, for some $i \in [3]$ there exists a path on vertices $x_1,x_2,\ldots,x_i \in V(T)\setminus V(P)$ such that $x_jx_{j+1} \in E(T)$ for $j \in [i-1]$ and $v_3x_1 \in E(T)$. If $i=3$, then by Observation~\ref{obs-ABC}~(i) and~(iv), $x_1\in D$. Thus, $d_T(x_1,v_2) = 2$, contradicting Definition~\ref{def-ABC}~(iii). If $i=2$, then $d_T(v_0,x_2) = 5$, contradicting Observation~\ref{obs-ABC}~(iii). Thus, $i=1$. Let $T'=T-\{x_1\}$. Recall that $v_2 \in D$ and observe that $D$ is an isolating set of $T'$.
    Using similar arguments as in the case for $d_T(v_1) \geq 3$ we can deduce that $D$ is an $\iota(T')$-set. Hence, $\iota(T) = \iota(T')$ and $(D,N_{T'}(D),V(T')\setminus N_{T'}[D])$ is a critical tripartition of $T'$ for every $\iota(T)$-set $D$. By Theorem~\ref{thm-ABC}, $T'$ is $(\iota,1)$-critical. Thus, by the induction hypothesis, $T' \in \Fi$. By Observation~\ref{obs-fiota}~(ii), $v_3$ has status $B$, and $T$ can be obtained by applying $\mathcal{O}_1$ to $T'$ on $v_3$. Therefore, $T \in \Fi$.

    Hence, $d_T(v_3) = 2$. Since $P$ is a diametral path with $d_T(v_d) = 1$, then by symmetry we have $d(v_{d-i}) = 2$ for every $i \in [3]$. Since $n \geq 6$ and $d(v_0,v_d)$ is even by Observation~\ref{obs-ABC}~(iii), $d \geq 8$. Let $T' = T - \{v_0,v_1,v_2,v_3\}$. We claim that $\iota(T') = \iota(T) - 1$. Observe that $D\setminus\{v_2\}$ is an isolating set of $T'$. Suppose for contradiction that there exists an $\iota(T')$-set $|D'|$ with $|D'| < |D\setminus \{v_2\}|$. In this case, $D'\cup \{v_2\}$ would be an isolating set of $T$ such that $|D'\cup \{v_2\}|< \iota(T)$, a contradiction. Hence, $\iota(T) = \iota(T')+1$ and $(D\setminus \{v_2\},N_{T'}(D\setminus \{v_2\}),V(T')\setminus N_{T'}[D\setminus \{v_2\}])$ is a critical tripartition of $T'$ for every $\iota(T)$-set $D$. By Theorem~\ref{thm-ABC}, $T'$ is $(\iota,1)$-critical. Thus, by the induction hypothesis, $T' \in \Fi$. If $d_{T'}(v_4) = 1$, then $v_4$ has status $C$. Suppose $d_{T'}(v_4) \geq 2$. Thus, for some $i \in [4]$ there exists a path on vertices $x_1,x_2,\ldots,x_i \in V(T')\setminus V(P)$ such that $x_jx_{j+1} \in E(T')$ for $j \in [i-1]$ and $v_4x_1 \in E(T')$. If $i \in \{1,3\}$, then $d(v_0,x_i)$ is odd, contradicting Observation~\ref{obs-ABC}~(iii). If $i = 2$, then $v_4 \in D$, contradicting Definition~\ref{def-ABC}~(iii). Thus, $i = 4$ and $v_4$ has status $C$. Hence, for both the case $d_{T'}(v_4) = 1$ and the case $d_{T'}(v_4) \geq 2$ with $i = 4$, $T$ can be obtained by applying $\mathcal{O}_3$ to $T'$ on $v_4$. Therefore, $T \in \Fi$.\qed{}

}

\begin{wn}
    If $T$ is an $(\iota,1)$-critical tree, then $T$ has exactly one $\iota(T)$-set.
\end{wn}
\proof{
    This follows from Theorem~\ref{thm-i1cons} and Proposition~\ref{prop-unique}.\qed{}
}

\begin{thm}\label{thm-fiota}
    A tree $T$ is $(\iota,1)$-critical if and only if $T \in \Fi$.
\end{thm}
\proof{The necessary condition holds by Theorem~\ref{thm-i1cons}. We now prove the sufficient condition.
    Let $T \in \mathcal{F}_\iota$ and let $\mathcal{A}_T$, $\mathcal{B}_T$, and $\mathcal{C}_T$ be the subsets of vertices in $V(T)$ with status $A$, $B$, and $C$, respectively. By Observation~\ref{obs-fiota}, $(\mathcal{A}_T, \mathcal{B}_T, \mathcal{C}_T)$ is a critical tripartition of $T$.    
    By Proposition~\ref{prop-unique}, $T$ has a unique $\iota(T)$-set $D_T$. By Observation~\ref{obs-fiota}~(iii) and~(vi) and the uniqueness of $D_T$, we have that $D_T=\mathcal{A}_T$, $N(D_T)=\mathcal{B}_T$ and $V\setminus N[D_T]=\mathcal{C}_T$. Therefore, by Theorem~\ref{thm-ABC}, $T$ is $(\iota,1)$-critical.\qed{} 
}

\subsection{Algorithmic aspects.} In 2025, Chen et al.~\cite{ChLWX25} gave a linear time algorithm for determining a minimum isolating set of a tree (among other interesting algorithmic results). On the other hand, a set $D$ is a unique isolating set of a tree $T$ if and only if each vertex $v \in D$ has at least two private edges $e,f$ such that $d_T(v,e) = d_T(v,f) = 1$ and $d_T(e,f) = 2$~\cite{sentUVED}. Therefore, once we compute any minimum isolating set of a tree $T$ in linear time, we can verify in linear time whether $T$ has a unique $\iota(T)$-set, which is a necessary condition for $T$ to be $(\iota,1)$-critical. 

Suppose that $D$ is a unique minimum isolating set of $T$. Since, for a given graph $G$ and three non-empty sets $A, B$ and $C$ that partition $V(G)$, we can verify in linear time whether $(A, B, C)$ is a critical tripartition of $G$ (keeping in mind Theorem~\ref{thm-ABC}) we immediately obtain the following corollary. 

\begin{wn}
    Given an $n$-vertex tree $T$, it can be verified in linear time whether $T$ is $(\iota,1)$-critical or not.
\end{wn}

By the same arguments, for any $n$-vertex graph $G$, if we can enumerate all $\iota(G)$-sets in time $T(n)$, then we can verify in $O(T(n) \cdot n)$ time whether $G$ is $(\iota,1)$-critical. Recall that in 2008, Fomin et al.~\cite{FGPS08} gave an $O(1.7159^n)$ time algorithm for enumerating all minimal dominating sets of an $n$-vertex graph, thereby showing that the maximum number of minimal dominating sets of such a graph is at most $1.7159^n$. On the other hand, for several classes of graphs, there exist better upper bounds, see for example~\cite{BDMN26,CHHK13,CLL15}. To the best of our knowledge, there is no such general result for enumerating all minimal isolating sets of a graph. 

\section{Open problem}

Observe that the $d$-wounded spider $S_{t,t-d}$ discussed in Section~\ref{sec:general} has $n=2t-d+1=2(t-d)+d+1$ vertices and is $(\iota,q)$-critical for $q=d+1=n-2(t-d)$. In particular, $q \neq n-3$, $q \neq n-5$, etc. Following a similar phenomenon in~\cite[Theorem 2.4]{DHLRT25}, we make the following conjecture.
\begin{conj} 
    For some values of $k$, there are no $(\iota,m-2k)$-critical graphs.
\end{conj} 
\noindent Using a computer, we verified this hypothesis for all $n$-vertex trees with $n \leq 16$.

\section*{Acknowledgements}
Peter Borg was supported by the SEA-EU Research Seed Fund 2025 grant 2025{\textunderscore}UM{\textunderscore}SEA-EU Seed{\textunderscore}01 of the University of Malta. The research work disclosed in this publication is partially funded by the Tertiary Education Scholarships Scheme (Malta).

\end{document}